\documentclass{ssl}
\usepackage{graphicx} 
\usepackage[toc,page]{appendix}
\usepackage[english]{babel}
\usepackage{amsmath, amsfonts, amsthm, amssymb}
\usepackage{array,booktabs, graphicx, tabularx, ,pdflscape, siunitx, makecell}

\title{Geometric Factorization of Sufficient Harmonic Representations}
\subtitle{Machine unlearning is defined by its performance, but this leaves another component of the model unexamined: its geometry.}
\author{Kennon Stewart\affilnum{1,2}}
\affiliation{\affilnum{1} Second Street Labs, Detroit, MI, USA\\
\affilnum{2} Department of Statistics, University of Michigan, Ann Arbor, MI, USA}
\corrauth{Kennon Stewart, Second Street Labs}
\email{kennon@secondstreetlabs.io}
\date{June 2026}
\version{v1}
\preprintnumber{2026-003}
\lablogo{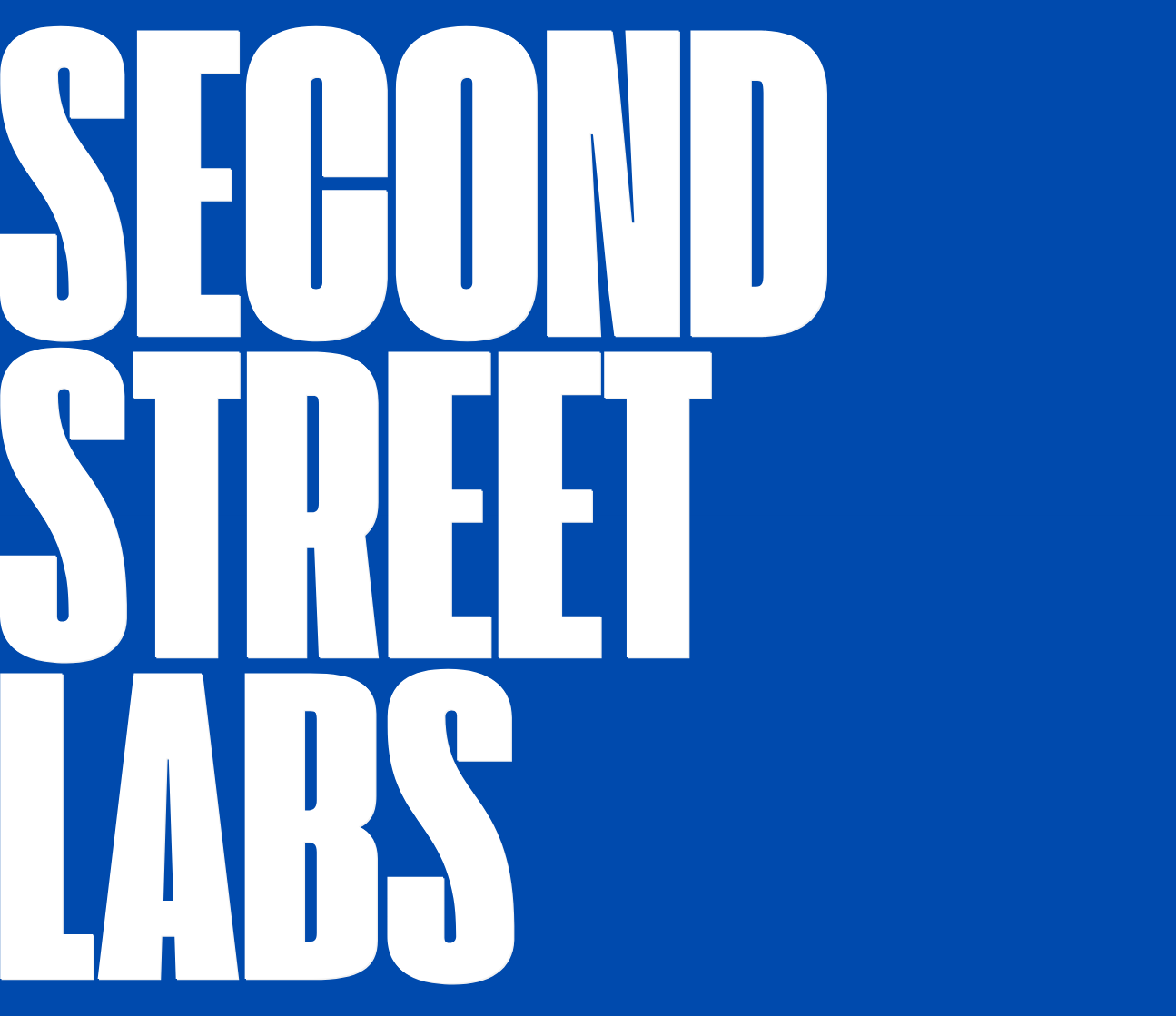}

\begin{document}
\begin{abstract}
For tasks of likelihood families invariant under the action of a lie group, the quotient is the minimal sufficient invariant representation. On compact homogeneous spaces, this quotient representation admits a harmonic realization through K-spherical Fourier coefficients; for finite-band harmonic exponential families, the empirical harmonic coefficients are minimal sufficient statistics. The partition function can be expressed algebraically by extracting the trivial representation component through Clebsch–Gordan decomposition.
\end{abstract}
\maketitle

\section{Introduction}

Representation learning studies how high-dimensional observations can be mapped to lower-dimensional coordinates that preserve the information needed for a task. This premise appears throughout modern machine learning: images, graphs, sequences, and physical measurements often contain repeated structure, nuisance variation, and symmetries that should not change the relevant prediction. A useful representation should therefore discard irrelevant variation while retaining the information required for inference.

Classical statistics gives a precise language for this goal. A statistic is sufficient for a parameter when it preserves all information in the data relevant to estimating that parameter. It is minimally sufficient when no strictly coarser statistic preserves the same information. In representation learning, the analogous question is geometric: when can a learned representation be said to contain exactly the task-relevant structure of the data, and no more?

This paper studies that question for data with group-structured symmetries. Suppose a group $G$ acts on a sample space $X$. Points in the same orbit represent different observations related by the symmetry action. If the task is invariant under this action, then these observations should not be distinguished by a minimal representation. The natural candidate is the quotient map $\pi:X\to X/G,$ which collapses each group orbit to a single equivalence class. We show that, under an invariant likelihood or conditional target law, this quotient is not merely an invariant representation: it is a sufficient representation. Under an additional orbit-separation condition, it is minimally sufficient.

This connects statistical sufficiency to geometric factorization. The quotient $X/G$ removes precisely the variation generated by the group action, while retaining the variation specifically relevant to the task at-hand. In this sense, minimal representation learning can be viewed as the problem of identifying and collapsing nuisance symmetries without destroying information needed for prediction.

We then specialize this idea to compact Lie groups and homogeneous spaces, where harmonic analysis provides canonical coordinates for group-structured data. By the Peter-Weyl theorem, square-integrable functions on a compact group decompose into matrix coefficients of irreducible unitary representations. For finite-band harmonic exponential families, these generalized Fourier coefficients become the natural sufficient statistics. We prove that, under the standard full-rank conditions for exponential families, the empirical generalized Fourier coefficients are minimally sufficient for the model parameters.

Finally, we study the normalization constant of these harmonic exponential families. Although finite-band log-densities are tractable to write down, their exponentials generally generate higher-order harmonics, making the partition function difficult to compute. We show that the Clebsch-Gordan decomposition expresses the normalization constant algebraically: after expanding the exponential, tensor products of irreducible representations decompose into direct sums, and Haar integration retains only the trivial representation component. Thus the partition function can be represented as the trivial-irrep projection of the exponential tensor algebra.

\subsection{Contributions}
The paper therefore makes three contributions. First, it formulates minimal sufficient representation learning as a quotient problem under group-invariant statistical structure. Second, it connects this quotient view to harmonic sufficient statistics on compact groups and homogeneous spaces. Third, it gives an algebraic representation of the harmonic exponential-family partition function using Clebsch-Gordan decomposition. Together, these results clarify when invariance produces compression, when equivariant structure remains necessary, and how minimal statistical sufficiency can be expressed in geometric and harmonic terms.

\section{Related Work}
\subsection{Exponential Families and Sufficient Statistics}
Exponential families provide the classical statistical setting in which sufficiency is most explicit. Given a dominated family of distributions $p_\theta(x)$, a statistic $T(X)$ is sufficient for $\theta$ if the likelihood depends on the sample only through $T(X)$ \cite{casella_statistical_2024, hipp_sufficient_1974}. In a full-rank exponential family, the empirical sum of the natural statistics is not only sufficient but minimally sufficient under standard regularity conditions. This makes exponential families a natural starting point for defining what it means to compress data without losing information relevant to an inferential task.

This paper builds on that idea but shifts the object of compression. Classical sufficiency begins with a specified likelihood and asks which statistic preserves information about the parameter. Representation learning often begins instead with high-dimensional observations whose relevant variation is structured by symmetries \cite{rosebrock_visual_2024}. In this setting, the central question is not only which statistic is sufficient for a parameter, but which symmetries remove the most nuisance variation while preserving task-specific structure. 

We formalize this connection by studying group actions on the sample space. If a group $G$ acts on $X$, then each orbit $Gx$ collects observations that differ only by the symmetry action. When the likelihood or conditional target law is invariant along these orbits, the quotient map $\pi:X\to X/G$ is a sufficient representation. Under an orbit-separation condition, it is minimally sufficient. Thus, the orbit space gives a geometric analogue of the minimal sufficient statistic: it is the coarsest representation that preserves the information relevant to the invariant task.

\subsection{Information Geometry}
Information geometry studies statistical models as manifolds whose points are probability distributions. In this view, learning takes place on a parametric manifold, and the Fisher information metric describes the local geometry of distinguishability among nearby distributions \cite{amari_information_2016, chirikjian_engineering_2000}. Chentsov's theorem gives a foundational invariance result: under suitable conditions \cite{dowty_chentsovs_2017, halverson_naturalness_2026}, the Fisher metric is the classical Riemannian structure preserved by sufficient statistic-like Markov morphisms.

Our perspective is complementary. Rather than beginning with the manifold of parametric distributions, we begin with the geometry of the sample space itself. When the sample space carries a group action, the relevant geometric operation is quotienting by the orbits of nuisance symmetries \cite{rosebrock_visual_2024,serre_linear_1977, sagan_symmetric_2001}. This converts a symmetry structure on data into a candidate sufficient representation. The resulting quotient $X/G$ is not a statistical manifold of distributions, but a reduced observation space whose points encode task-relevant equivalence classes.

This distinction is important for representation learning. In many machine learning settings, the parametric family is implicit, overparameterized, or difficult to interpret directly. The geometry of the data and the symmetries of the task may be more accessible than the geometry of the learned model's full parameter space. Our contribution is to connect these two views: under invariant statistical structure, geometric quotienting induces statistical sufficiency.

\subsection{Geometric Deep Learning and Invariant Representations}
Geometric deep learning studies learning problems with known or hypothesized symmetries. A representation is invariant if it is unchanged under a group action, and equivariant if it transforms predictably when the input is transformed \cite{rosebrock_visual_2024}. These principles have led to architectures for images, graphs, sets, manifolds, and physical systems, where respecting symmetry can improve generalization and reduce sample complexity.

Our work differs in emphasis. Much of geometric deep learning asks how to design architectures that respect a symmetry. We ask when respecting a symmetry produces a minimal sufficient representation. The answer is task-dependent. For invariant tasks, variation along group orbits is nuisance variation, and a minimal representation should collapse each orbit. For equivariant or reconstructive tasks, however, nontrivial group coordinates may be necessary and should not be treated as redundancy.

This distinction clarifies the role of irreducible representations. Decomposing a representation into irreducible components identifies the elementary harmonic modes through which the group acts \cite{greiner_quantum_1989, hall_lie_2015}. But irreducibility alone is not the same as minimal sufficiency. A nontrivial irreducible component may be redundant for an invariant classification task and essential for an equivariant prediction task. Minimality is therefore not a property of the representation alone; it is a property of the representation relative to the statistical task.

\subsection{Harmonic Analysis and Minimal Irreducible Coordinates}
For compact groups and homogeneous spaces, harmonic analysis provides canonical coordinates for functions on the space. The Peter-Weyl theorem decomposes square-integrable functions on a compact group into matrix coefficients of irreducible unitary representations \cite{kondor_clebsch-gordan_2018, donkin_clebsch-gordan_2020}. These matrix coefficients generalize ordinary Fourier modes to noncommutative groups.

This harmonic viewpoint connects directly to sufficiency for exponential families on compact groups. If a harmonic exponential family is defined by a finite set of irreducible representation coefficients \cite{cohen_harmonic_2015, tsipidi_harmonic_2025,sundararajan_fourier_2018}, then the corresponding empirical generalized Fourier coefficients are the natural sufficient statistics. Under full-rank conditions, these statistics are minimally sufficient for the model parameters. Thus, harmonic coefficients provide a finite-band statistical realization of the more general quotient-sufficiency principle.

The quotient and harmonic views play different roles. The quotient $X/G$ identifies the minimal invariant representation with respect to the sample space. The harmonic decomposition provides coordinates for representing functions, densities, and finite-band exponential families on compact groups or homogeneous spaces. Our aim is to connect these levels: orbit spaces describe what invariant representations should collapse, while irreducible harmonic coefficients describe the shape of those irreducibles.

\section{Theoretical Results}
This section develops the connection between minimal statistical sufficiency, orbit spaces, and harmonic coordinates on compact groups. We do this in four steps. First, we show that when a likelihood or prediction task is invariant under a group action, the orbit quotient is a sufficient representation. We then use standard statistical theory to define the conditions for minimal sufficiency. Second, we analyze compact Lie groups, where Peter-Weyl theory provides harmonic coordinates through irreducible unitary representations. Third, we show that for finite-band harmonic exponential families, empirical generalized Fourier coefficients are minimally sufficient statistics. Finally, we express the partition function of these families algebraically using Clebsch-Gordan decomposition.

\subsection{Minimal Sufficient Invariant Quotients}
Let $G$ be a compact group acting measurably on a sample space $\mathcal X$, and let
$$
\pi:\mathcal X\to \mathcal X/G
$$
denote the quotient map sending each point to its orbit. The quotient identifies observations that differ only by the group action.

\begin{theorem}[Sufficiency of the Orbit Quotient]
Let ${P_\theta:\theta\in\Theta}$ be a dominated statistical family on $\mathcal X$ with density $p_\theta(x)$. Suppose the likelihood is invariant along group orbits:
$$
p_\theta(gx)=p_\theta(x)
\quad
\text{for all } g\in G,\ x\in\mathcal X,\ \theta\in\Theta.
$$
Then $\pi(X)$ is sufficient for $\theta$.
\end{theorem}

\begin{proof}
Let \(T=\pi\). Since \(p_\theta(gx)=p_\theta(x)\) for all \(g\in G\), the density \(p_\theta\) is constant on each orbit of the group action.

Define \(q_\theta:\mathcal X/G\to \mathbb R\) by
\[
q_\theta([x]) = p_\theta(x),
\]
where \([x]=\pi(x)\) denotes the orbit of \(x\).

This definition is well-defined. This means that if \([x]=[x']\), then \(x'=gx\) for some \(g\in G\), and therefore
\[
p_\theta(x')=p_\theta(gx)=p_\theta(x).
\]
And so \(q_\theta([x])\) is independent of the representative chosen from the orbit.

Therefore,
\[
p_\theta(x)=q_\theta(\pi(x)),
\]
which we can also write as 
\[
p_\theta(x)=q_\theta(T(x))\cdot 1.
\]
This is basically Fisher-Neyman factorization with \(h(x)=1\), where all dependence on \(\theta\) occurs through \(T(x)=\pi(x)\). As such, \(T(X)=\pi(X)\) is sufficient for \(\theta\).
\end{proof}

\begin{theorem}[Minimal Sufficient Orbit Quotient]
Assume, in addition, that distinct orbits are statistically distinguishable:
$$
\pi(x)\neq \pi(x')
\quad\Longrightarrow\quad
\frac{p_\theta(x)}{p_\theta(x')}
\text{ is not constant in }\theta.
$$
Then $\pi(X)$ is minimal sufficient for $\theta$.
\end{theorem}

\begin{proof}
By the likelihood-ratio characterization of minimal sufficiency, a statistic
\(T(X)\) is minimal sufficient if
\[
T(x)=T(x')
\quad\Longleftrightarrow\quad
\frac{p_\theta(x)}{p_\theta(x')}
\text{ is constant in }\theta,
\]
up to the usual qualifications on null sets and points where the denominator
vanishes.

Let \(T=\pi\). Suppose first that \(\pi(x)=\pi(x')\). Then \(x\) and \(x'\)
belong to the same orbit, so there exists \(g\in G\) such that \(x'=gx\).
By orbit invariance,
\[
p_\theta(x')=p_\theta(gx)=p_\theta(x)
\]
for all \(\theta\in\Theta\). Hence
\[
\frac{p_\theta(x)}{p_\theta(x')}=1,
\]
which is constant in \(\theta\).

Conversely, suppose that
\[
\frac{p_\theta(x)}{p_\theta(x')}
\]
is constant in \(\theta\). By the orbit-separation assumption, distinct orbits
cannot have a likelihood ratio that is constant in \(\theta\). Therefore it
cannot be the case that \(\pi(x)\neq \pi(x')\). Hence
\[
\pi(x)=\pi(x').
\]

Thus
\[
\pi(x)=\pi(x')
\quad\Longleftrightarrow\quad
\frac{p_\theta(x)}{p_\theta(x')}
\text{ is constant in }\theta.
\]
By the likelihood-ratio characterization of minimal sufficiency, \(\pi(X)\)
is minimal sufficient for \(\theta\).
\end{proof}

This result formalizes the sense in which an orbit space is a minimal invariant representation. If the task cannot distinguish points within the same orbit, then orbit-level information is sufficient. If distinct orbits remain statistically distinguishable, then no coarser invariant representation is sufficient.

\subsection{Generalized Fourier Statistics on Compact Groups}

We now specialize to the case where the sample space is a compact Lie group $G$ equipped with normalized Haar measure $dg$. Let $X_1,\dots,X_n$ be i.i.d. random variables taking values in $G$, with density $p\in L^2(G)$.

Let $\widehat G$ denote the set of equivalence classes of irreducible unitary representations of $G$. For $\lambda\in\widehat G$, let
$$
\rho_\lambda:G\to U(d_\lambda)
$$
be an irreducible unitary representation of dimension $d_\lambda$. The generalized Fourier coefficient of $p$ at frequency $\lambda$ is
$$
\widehat p(\lambda)
=
\int_G p(g)\rho_\lambda(g)^\dagger,dg.
$$
Given observations $X_1,\dots,X_n$, its empirical estimate is
$$
\widehat p_{\mathrm{emp}}(\lambda)
=
\frac{1}{n}\sum_{i=1}^n \rho_\lambda(X_i)^\dagger.
$$

For a finite subset $\widehat G_N\subset\widehat G$, define the Peter-Weyl projection
$$
(\mathcal T_Np)(g)
=
\sum_{\lambda\in\widehat G_N}
d_\lambda
\operatorname{Tr}!\left(
\widehat p(\lambda)\rho_\lambda(g)
\right).
$$
This is a finite-band approximation of $p$, retaining only the harmonic components indexed by $\widehat G_N$.

\subsection{Minimal Sufficiency of Finite-Band Harmonic Statistics}

Consider the finite-band harmonic exponential family
$$
p_\theta(g)
=
\exp\left[
E_\theta(g)-A(\theta)
\right],
$$
where
$$
E_\theta(g)
=
\sum_{\lambda\in\widehat G_N}
d_\lambda
\operatorname{Re}
\operatorname{Tr}
\left(
C_\lambda^\dagger\rho_\lambda(g)
\right),
$$
and $\theta={C_\lambda}{\lambda\in\widehat G_N}$.

\begin{theorem}[Minimal Sufficiency of Harmonic Statistics]
Suppose the finite-band harmonic exponential family above is full rank: after choosing real coordinates for the matrix coefficients, the natural statistic has no nontrivial affine dependence almost surely, and the natural parameter space contains an open set. Then, for i.i.d. observations $X_1,\dots,X_n\in G$, the statistic
$$
T_N(X_1,\dots,X_n)
=
\left(
\sum{i=1}^n \rho_\lambda(X_i)^\dagger
\right){\lambda\in\widehat G_N}
$$
is minimal sufficient for $\theta$.
\end{theorem}

\begin{proof}
For i.i.d. observations \(X_1,\dots,X_n\in G\), the joint likelihood is
\[
L(\theta\mid X_1,\dots,X_n)
=
\prod_{i=1}^n p_\theta(X_i)
=
\exp\left[
\sum_{i=1}^n E_\theta(X_i)-nA(\theta)
\right].
\]
Substituting the finite-band harmonic energy gives
\[
\sum_{i=1}^n E_\theta(X_i)
=
\sum_{i=1}^n
\sum_{\lambda\in\widehat G_N}
d_\lambda
\operatorname{Re}
\operatorname{Tr}
\left(
C_\lambda^\dagger \rho_\lambda(X_i)
\right).
\]
Since the index set \(\widehat G_N\) is finite, we may interchange the sums and use
linearity of the trace:
\[
\sum_{i=1}^n E_\theta(X_i)
=
\sum_{\lambda\in\widehat G_N}
d_\lambda
\operatorname{Re}
\operatorname{Tr}
\left(
C_\lambda^\dagger
\sum_{i=1}^n \rho_\lambda(X_i)
\right).
\]
Therefore
\[
L(\theta\mid X_1,\dots,X_n)
=
\exp\left[
\sum_{\lambda\in\widehat G_N}
d_\lambda
\operatorname{Re}
\operatorname{Tr}
\left(
C_\lambda^\dagger T_{N,\lambda}(X)
\right)
-
nA(\theta)
\right],
\]
where
\[
T_{N,\lambda}(X)=\sum_{i=1}^n \rho_\lambda(X_i).
\]
Thus the likelihood depends on the sample only through
\[
T_N(X)=
\left(
T_{N,\lambda}(X)
\right)_{\lambda\in\widehat G_N},
\]
and the Fisher-Neyman factorization theorem implies that \(T_N\) is sufficient.

It remains to show minimality. Let \(X=(X_1,\dots,X_n)\) and
\(Y=(Y_1,\dots,Y_n)\). The likelihood ratio is
\[
\frac{L(\theta\mid X)}{L(\theta\mid Y)}
=
\exp\left[
\sum_{\lambda\in\widehat G_N}
d_\lambda
\operatorname{Re}
\operatorname{Tr}
\left(
C_\lambda^\dagger
\left[
T_{N,\lambda}(X)-T_{N,\lambda}(Y)
\right]
\right)
\right],
\]
since the terms \(nA(\theta)\) cancel.

If \(T_N(X)=T_N(Y)\), then the exponent is zero and the likelihood ratio is
equal to \(1\), hence constant in \(\theta\).

Conversely, suppose the likelihood ratio is constant in \(\theta\). Then
\[
\sum_{\lambda\in\widehat G_N}
d_\lambda
\operatorname{Re}
\operatorname{Tr}
\left(
C_\lambda^\dagger
\left[
T_{N,\lambda}(X)-T_{N,\lambda}(Y)
\right]
\right)
\]
is constant as a function of the natural parameters
\((C_\lambda)_{\lambda\in\widehat G_N}\). Because the natural parameter space
contains an open set and the family is full rank after choosing real
coordinates, this is possible only if
\[
T_{N,\lambda}(X)-T_{N,\lambda}(Y)=0
\quad
\text{for all }
\lambda\in\widehat G_N.
\]
Hence \(T_N(X)=T_N(Y)\).

Therefore
\[
T_N(X)=T_N(Y)
\quad\Longleftrightarrow\quad
\frac{L(\theta\mid X)}{L(\theta\mid Y)}
\text{ is constant in }\theta.
\]
By the likelihood-ratio characterization of minimal sufficiency,
\(T_N\) is minimal sufficient.
\end{proof}

\subsection{Algebraic Normalization by Clebsch-Gordan Decomposition}

The partition function is
$$
e^{A(\theta)}
=
\int_G e^{E_\theta(g)},dg.
$$
Although $E_\theta$ is finite-band, $e^{E_\theta}$ generally is not. The following result expresses the partition function as an exact algebraic expansion.

\begin{theorem}[Clebsch-Gordan Normalization]
For the finite-band harmonic exponential family above,
$$
e^{A(\theta)}
=
\sum_{k=0}^{\infty}
\frac{1}{k!}
\operatorname{Coeff}{\mathbf 1}
\left(E\theta^k\right),
$$
where $\operatorname{Coeff}{\mathbf 1}(E\theta^k)$ is the coefficient of the trivial representation in the Peter-Weyl expansion of $E_\theta(g)^k$, obtained by decomposing tensor products of irreducible representations using Clebsch-Gordan rules.
\end{theorem}

\begin{proof}
Since \(G\) is compact and \(E_\theta\) is continuous, \(E_\theta\) is bounded.
Hence the exponential series
\[
e^{E_\theta(g)}
=
\sum_{k=0}^{\infty}
\frac{E_\theta(g)^k}{k!}
\]
converges uniformly on \(G\). Therefore we may interchange summation and
integration:
\[
e^{A(\theta)}
=
\int_G e^{E_\theta(g)}\,dg
=
\sum_{k=0}^{\infty}
\frac{1}{k!}
\int_G E_\theta(g)^k\,dg.
\]

For each power, we then need to identify the integral. The energy \(E_\theta\) is a
finite linear combination of matrix coefficients of irreducible unitary
representations. Therefore \(E_\theta^k\) is a finite linear combination of
products of \(k\) such matrix coefficients. Each such product is a matrix
coefficient of the tensor-product representation
\[
\rho_{\lambda_1}\otimes\cdots\otimes\rho_{\lambda_k}.
\]
By the Clebsch-Gordan decomposition, this tensor product decomposes as a
finite direct sum of irreducible representations:
\[
\rho_{\lambda_1}\otimes\cdots\otimes\rho_{\lambda_k}
\cong
\bigoplus_{\nu\in\widehat G}
N_{\lambda_1,\dots,\lambda_k}^{\nu}\rho_\nu.
\]
Thus \(E_\theta^k\) admits a Peter-Weyl expansion into irreducible components.

Haar integration is the projection onto the invariant subspace. Equivalently,
by Peter-Weyl orthogonality, the integral of every nontrivial irreducible
matrix coefficient vanishes, while the coefficient of the trivial
representation is preserved. Hence
\[
\int_G E_\theta(g)^k\,dg
=
\operatorname{Coeff}_{\mathbf 1}(E_\theta^k).
\]
Substituting this identity into the exponential expansion yields
\[
e^{A(\theta)}
=
\sum_{k=0}^{\infty}
\frac{1}{k!}
\operatorname{Coeff}_{\mathbf 1}(E_\theta^k).
\]
\end{proof}

This theorem replaces continuous integration over $G$ with extraction of the invariant component in the tensor algebra generated by the finite set of harmonics. It is an exact representation, though practical computation requires truncating the infinite series and controlling the resulting error.

\subsection{Peter-Weyl Recovery}

\begin{theorem}[Harmonic Recovery]
Let $p\in L^2(G)$. The algebraic span of the matrix coefficients of irreducible unitary representations of $G$ is dense in $L^2(G)$. Therefore, for any sequence of finite subsets $\widehat G_N$ exhausting $\widehat G$,
$$
\lim_{N\to\infty}
|p-\mathcal T_Np|_{L^2(G)}
=
0.
$$
\end{theorem}

This result establishes harmonic completeness. It shows that no $L^2$ information is lost in the infinite-band limit. The sufficiency result above is instead a finite-band statement about a specified exponential family.

\section{Discussion}
This paper connects classical sufficiency, group quotients, and harmonic representation theory. Its central claim is that minimal sufficient representation learning can be understood as geometric factorization: when the task is invariant under a group action, the orbit quotient is the natural minimal sufficient invariant. For compact groups and homogeneous spaces, this quotient structure admits harmonic coordinates through irreducible representations. In finite-band harmonic exponential families, the empirical generalized Fourier coefficients become minimally sufficient statistics, while Clebsch-Gordan decomposition gives an algebraic representation of the partition function through the trivial representation component.

\bibliographystyle{plain}
\bibliography{references}

@book{casella_statistical_2024,
    address = {Boca Raton},
    edition = {2},
    title = {Statistical {Inference}},
    isbn = {9781003456285},
    url = {https://www.taylorfrancis.com/books/9781003456285},
    language = {en},
    urldate = {2025-12-27},
    publisher = {Chapman and Hall/CRC},
    author = {Casella, George and Berger, Roger},
    month = apr,
    year = {2024},
    doi = {10.1201/9781003456285},
}

@book{amari_information_2016,
    address = {Tokyo},
    series = {Applied {Mathematical} {Sciences}},
    title = {Information {Geometry} and {Its} {Applications}},
    volume = {194},
    copyright = {https://www.springernature.com/gp/researchers/text-and-data-mining},
    isbn = {9784431559771 9784431559788},
    url = {https://link.springer.com/10.1007/978-4-431-55978-8},
    language = {en},
    urldate = {2026-06-04},
    publisher = {Springer Japan},
    author = {Amari, Shun-ichi},
    year = {2016},
    doi = {10.1007/978-4-431-55978-8},
}

@book{chirikjian_engineering_2000,
    edition = {0},
    title = {Engineering {Applications} of {Noncommutative} {Harmonic} {Analysis}: {With} {Emphasis} on {Rotation} and {Motion} {Groups}},
    isbn = {9780429123511},
    shorttitle = {Engineering {Applications} of {Noncommutative} {Harmonic} {Analysis}},
    url = {https://www.taylorfrancis.com/books/9781420041767},
    language = {en},
    urldate = {2026-06-04},
    publisher = {CRC Press},
    author = {Chirikjian, Gregory S.},
    month = sep,
    year = {2000},
    doi = {10.1201/9781420041767},
}

@article{cohen_harmonic_2015,
    title = {Harmonic {Exponential} {Families} on {Manifolds}},
    copyright = {arXiv.org perpetual, non-exclusive license},
    url = {https://arxiv.org/abs/1505.04413},
    doi = {10.48550/ARXIV.1505.04413},
    urldate = {2026-06-04},
    author = {Cohen, Taco S. and Welling, Max},
    year = {2015},
}

@article{hipp_sufficient_1974,
    title = {Sufficient {Statistics} and {Exponential} {Families}},
    volume = {2},
    issn = {0090-5364, 2168-8966},
    url = {https://projecteuclid.org/journals/annals-of-statistics/volume-2/issue-6/Sufficient-Statistics-and-Exponential-Families/10.1214/aos/1176342879.full},
    doi = {10.1214/aos/1176342879},
    language = {en},
    number = {6},
    urldate = {2026-06-04},
    journal = {The Annals of Statistics},
    author = {Hipp, Christian},
    month = nov,
    year = {1974},
    pages = {1283--1292},
}

@misc{dowty_chentsovs_2017,
    title = {Chentsov's theorem for exponential families},
    url = {https://arxiv.org/abs/1701.08895v2},
    language = {en},
    urldate = {2026-06-04},
    journal = {arXiv.org},
    author = {Dowty, James G.},
    month = jan,
    year = {2017},
}

@misc{halverson_naturalness_2026,
    title = {Naturalness and {Fisher} {Information}},
    url = {http://arxiv.org/abs/2603.01411},
    doi = {10.48550/arXiv.2603.01411},
    urldate = {2026-06-04},
    publisher = {arXiv},
    author = {Halverson, James and Harvey, Thomas R. and Nee, Michael},
    month = may,
    year = {2026},
    note = {arXiv:2603.01411},
}

@book{rosebrock_visual_2024,
    address = {Berlin, Heidelberg},
    series = {Springer {Undergraduate} {Mathematics} {Series}},
    title = {Visual {Group} {Theory}: {A} {Computer}-{Oriented} {Geometric} {Introduction}},
    copyright = {https://www.springernature.com/gp/researchers/text-and-data-mining},
    isbn = {9783662693643 9783662693650},
    shorttitle = {Visual {Group} {Theory}},
    url = {https://link.springer.com/10.1007/978-3-662-69365-0},
    language = {en},
    urldate = {2026-06-04},
    publisher = {Springer Berlin Heidelberg},
    author = {Rosebrock, Stephan},
    year = {2024},
    doi = {10.1007/978-3-662-69365-0},
}

@book{serre_linear_1977,
    address = {New York, NY},
    series = {Graduate {Texts} in {Mathematics}},
    title = {Linear {Representations} of {Finite} {Groups}},
    volume = {42},
    copyright = {https://www.springernature.com/gp/researchers/text-and-data-mining},
    isbn = {9781468494600 9781468494587},
    url = {https://link.springer.com/10.1007/978-1-4684-9458-7},
    language = {en},
    urldate = {2026-06-03},
    publisher = {Springer},
    author = {Serre, Jean-Pierre},
    year = {1977},
    doi = {10.1007/978-1-4684-9458-7},
}

@book{sagan_symmetric_2001,
    address = {New York, NY},
    series = {Graduate {Texts} in {Mathematics}},
    title = {The {Symmetric} {Group}},
    volume = {203},
    copyright = {http://www.springer.com/tdm},
    isbn = {9781441928696 9781475768046},
    url = {http://link.springer.com/10.1007/978-1-4757-6804-6},
    urldate = {2026-06-04},
    publisher = {Springer New York},
    author = {Sagan, Bruce E.},
    year = {2001},
    doi = {10.1007/978-1-4757-6804-6},
}

@book{greiner_quantum_1989,
    address = {Berlin Heidelberg},
    title = {Quantum {Mechanics}: {Symmetries}},
    isbn = {9783662009024},
    shorttitle = {Quantum {Mechanics}},
    language = {eng},
    publisher = {Springer},
    author = {Greiner, Walter and Müller, Berndt},
    year = {1989},
}

@book{hall_lie_2015,
    address = {Cham},
    series = {Graduate {Texts} in {Mathematics}},
    title = {Lie {Groups}, {Lie} {Algebras}, and {Representations}: {An} {Elementary} {Introduction}},
    volume = {222},
    copyright = {https://www.springernature.com/gp/researchers/text-and-data-mining},
    isbn = {9783319134666 9783319134673},
    shorttitle = {Lie {Groups}, {Lie} {Algebras}, and {Representations}},
    url = {https://link.springer.com/10.1007/978-3-319-13467-3},
    language = {en},
    urldate = {2026-06-04},
    publisher = {Springer International Publishing},
    author = {Hall, Brian C.},
    year = {2015},
    doi = {10.1007/978-3-319-13467-3},
}

@misc{kondor_clebsch-gordan_2018,
    title = {Clebsch-{Gordan} {Nets}: a {Fully} {Fourier} {Space} {Spherical} {Convolutional} {Neural} {Network}},
    copyright = {arXiv.org perpetual, non-exclusive license},
    shorttitle = {Clebsch-{Gordan} {Nets}},
    url = {https://arxiv.org/abs/1806.09231},
    doi = {10.48550/ARXIV.1806.09231},
    urldate = {2026-06-04},
    publisher = {arXiv},
    author = {Kondor, Risi and Lin, Zhen and Trivedi, Shubhendu},
    year = {2018},
}

@article{donkin_clebsch-gordan_2020,
    title = {A {Clebsch}-{Gordan} decomposition in positive characteristic},
    volume = {560},
    issn = {00218693},
    url = {https://linkinghub.elsevier.com/retrieve/pii/S002186932030291X},
    doi = {10.1016/j.jalgebra.2020.06.001},
    language = {en},
    urldate = {2026-06-04},
    journal = {Journal of Algebra},
    author = {Donkin, Stephen and Martin, Samuel},
    month = oct,
    year = {2020},
    pages = {680--699},
}

@misc{tsipidi_harmonic_2025,
    title = {The {Harmonic} {Structure} of {Information} {Contours}},
    copyright = {Creative Commons Attribution 4.0 International},
    url = {https://arxiv.org/abs/2506.03902},
    doi = {10.48550/ARXIV.2506.03902},
    urldate = {2026-06-04},
    publisher = {arXiv},
    author = {Tsipidi, Eleftheria and Kiegeland, Samuel and Nowak, Franz and Xu, Tianyang and Wilcox, Ethan and Warstadt, Alex and Cotterell, Ryan and Giulianelli, Mario},
    year = {2025},
}

@book{sundararajan_fourier_2018,
    address = {Singapore},
    series = {{SpringerLink} {Bücher}},
    title = {Fourier {Analysis}—{A} {Signal} {Processing} {Approach}},
    isbn = {9789811316937},
    language = {eng},
    publisher = {Springer},
    author = {Sundararajan, D.},
    year = {2018},
}

\end{document}